\pdfoutput=1
\documentclass[11pt]{amsart}
\usepackage[paper,pkg={amsaddr=true,tikz=true,biblatex/opt={+giveninits=true}}]{zbs}
\newcommand{\FF}{\mathcal{F}}
\newcommand{\KK}{\mathcal{K}}

\title{Maximum shattering}
\date{}
\author{Noga Alon}
\email{nalon@math.princeton.edu}
\author{Varun Sivashankar}
\email{varunsiva@princeton.edu}
\author{Daniel G. Zhu}
\email{zhd@princeton.edu}
\address{Department of Mathematics, Princeton University, Princeton, NJ 08544, USA}
\thanks{The first author was supported by NSF grant DMS-2154082. The third author was supported by the NSF Graduate Research Fellowships Program (NSF grant DGE-2039656).}

\begin{document}
\begin{abstract}
A family $\mathcal{F}$ of subsets of $[n]=\{1,2,\ldots,n\}$ shatters a set $A \subseteq [n]$ if for every $A' \subseteq A$ there is an $F \in \mathcal{F}$ such that $F \cap A=A'$. We develop a framework to analyze $f(n,k,d)$, the maximum possible number of subsets of $[n]$ of size $d$ that can be shattered by a family of size $k$. Among other results, we determine $f(n,k,d)$ exactly for $d \leq 2$ and show that if $d$ and $n$ grow, with both $d$ and $n-d$ tending to infinity, then, for any $k$ satisfying $2^d \leq k \leq (1+o(1))2^d$, we have $f(n,k,d)=(1+o(1))c\binom{n}{d}$, where $c$, roughly $0.289$, is the probability that a large square matrix over $\mathbb{F}_2$ is invertible. This latter result extends work of Das and M\'esz\'aros. As an application, we improve bounds for the existence of covering arrays for certain alphabet sizes.
\end{abstract}

\maketitle

\section{Introduction}
Let $\FF \subseteq 2^{[n]}$ be a family of subsets of
$[n]=\{1,2, \ldots ,n\}$. 
A set $A \subseteq [n]$ is \emph{shattered} by $\FF$
if $\{A \cap S : S \in \FF \} = 2^{A}$, that is,
if for every subset $A' \subseteq A$ there exists some $F \in \FF$ 
such that $F \cap A = A'$. The well-known
Sauer-Perles-Shelah lemma \cite{Sa,She} states that 
if $|\FF| > \sum_{i=0}^{d-1} \binom{n}{i}$ then $\FF$ shatters at
least one set of size at least $d$. A slightly stronger result,
first proved by Pajor \cite{Paj} (see also \cite{ARS}), asserts that
any family $\FF$ shatters at least $|\FF|$ distinct subsets. This
implies the previous statement since if more than 
$ \sum_{i=0}^{d-1} \binom{n}{i}$ subsets are shattered, then at
least one of these subsets has size greater than $d-1$. Combining
this result with the Kruskal-Katona Theorem \cite{Kr,Ka},
it is possible to determine, for all $n$, $k$, and $d$, the 
minimum possible number of 
subsets of size $d$ of $[n]$ that are shattered by any 
family of $k$ distinct subsets of $[n]$ (see \cref{sec:minshatter} of this paper for more details).

Our objective in the present paper is to study the opposite question: 
given positive integers $n$, $k$, and $d$, what is the \emph{largest} 
number of 
subsets of size $d$ that a family of sets $\FF \subseteq 2^{[n]}$ of
size at most $k$ can shatter?\footnote{The sole purpose of specifying $\abs{\FF} \leq k$ instead of $\abs{\FF} = k$ is to prevent $f(n,k,d)$ from being undefined when $k > 2^n$.} Denote this number
by $f(n,k,d)$. It is clear that this number is $0$ if $k<2^d$ or $n < d$ 
and that for every fixed $d$ it is weakly increasing in $n$ and in $k$. Moreover,
\[f(n, k, 1) = \begin{cases}
	0 & k = 1 \\
	n & k \geq 2
\end{cases}\]
by
placing both $\emptyset$ and $[n]$ in $\FF$. 
Thus, the
interesting cases are those where
$d \geq 2$, $k \geq 2^d$, and $n\geq d$.

In this paper, we first demonstrate that the exact values of $f(n,k,2)$ follow from work of Kleitman and Spencer \cite{KS} on pairwise independent sets. In the case where $k = 2^d$, Das and M\'esz\'aros~\cite{DM18} obtained the following bound:
\begin{thm}[\cite{DM18}]\label{t12}
For any $n \geq d \geq 1$, we have
\[c_d \binom{n}{d} \leq f(n, 2^d, d) \leq \frac{c_d n^d}{d!},\]
where
\[c_d = \frac{(2^d-2)(2^d-4) \cdots (2^d-2^{d-1})}{(2^d-1)^{d-1}}\]
is the probability that $d$ independent uniformly random vectors in $\setf_2^d \setminus \set{0}$ are linearly independent. Moreover, if $n$ is a multiple of $2^d-1$, equality holds in the upper bound.
\end{thm}
Whenever $n \gg d^2$, the quantities $\binom{n}{d}$ and $\frac{n^d}{d!}$ are quite close, so \cref{t12} gives tight bounds. In particular, it implies that for fixed $d$,
\[f(n, 2^d, d) = (1 + o(1)) \frac{c_d n^d}{d!}.\]

For smaller values of $n$ compared to $d$, we can still get tight bounds at the expense of requiring that $d$ and $n-d$ grow.
\begin{thm}
	\label{t13}
	If $d$ and $n$ grow, with both $d$ and $n-d$ tending to infinity,
	then $f(n,2^d,d)=(1+o(1))c \binom{n}{d}$, where $c = \lim_{d \to\infty}c_d = \prod_{i=1}^\infty(1-2^{-i}) \approx 0.289$.
\end{thm}

To prove \cref{t13}, we develop a more general theory of the function $f(n,k,d)$ based on determining the Lagrangians of relevant hypergraphs, which also reproduces \cref{t12}. Although we do not have matching upper and lower bounds when $k > 2^d$, we are nonetheless able to prove structural results concerning continuity and the types of asymptotic growth encountered in various regimes. One particular consequence is the following strengthening of \cref{t13}:
\begin{thm}
	\label{t13a}
	If $d \geq 1$, $k \geq 2^d$, and $n\geq d$ are growing positive integers such that $d$ and $n-d$ tend to infinity and $k = (1 + o(1))2^d$, then
 $f(n,k,d)=(1+o(1))c \binom{n}{d}$.
\end{thm}

Finally, we connect the results of this paper to the theory of covering arrays. In addition to reproducing many results in the literature, we are able to improve the best-known bounds for the existence of covering arrays for certain alphabet sizes, the smallest of which are $35$, $40$, and $45$.

\subsection*{Outline} \cref{sec:general}
develops a framework for analyzing the asymptotics of $f(n,k,d)$. 
In \cref{sec:pairs} we derive the 
precise value of $f(n,k,2)$ for all $n$ and $k$ by
combining a result of Kleitman and Spencer with Tur\'an's Theorem. \cref{sec:k2d} deals with the case $k = 2^d$, and quickly derives \cref{t12,t13,t13a} from more general results.
Several bounds for the case $k>2^d$ are discussed in \cref{sec:kg2d}, and the final \cref{sec:final}
contains some concluding remarks, including the application to covering arrays.

\section{The Asymptotic Structure of \texorpdfstring{$f(n,k,d)$}{f(n, k, d)}} \label{sec:general}
\subsection{Shattering hypergraphs and Lagrangians}
A family of sets $\FF \subseteq 2^{[n]}$ of size at most $k$ can be
represented 
by a $k \times n$ binary matrix, where the rows are the
indicator vectors of the sets $S \in \FF$.
Say a $k \times d$ binary matrix is \emph{shattered} if each of
the $2^d$ possible rows appears among its $k$ rows. 
Then, $f(n,k,d)$ is the maximum
possible number
of shattered $k \times d$ submatrices of a binary $k \times n$
matrix.

With this interpretation in mind, we make the following definition, generalizating a construction in \cite{DM18}:
\begin{defn}
For integers $d \geq 1$ and $k \geq 2^d$, we define $H(k,d)$ to be the $d$-uniform hypergraph with vertex set $\set{0,1}^k$, i.e.\ the set of binary vectors of length $k$. A collection of $d$ such vectors forms an edge if and only if the $k \times d$ matrix with these vectors as columns is shattered.
\end{defn}

We also recall the definition of the Lagrangian of a hypergraph, first considered by Frankl and F\"uredi
\cite{FF} and by Sidorenko \cite{Si},
extending the application of this notion for graphs, initiated
by Motzkin and Straus \cite{MS}.
\begin{defn}
The \emph{Lagrangian polynomial} of a $d$-uniform hypergraph $H$ is the polynomial
\[
P_H((x_v)_{v \in H})=\sum_{e \in E(H)} \prod_{v \in e} x_j.
\]
The \emph{Lagrangian} $\lambda(H)$ of $H$ is the maximum value of
$P_H$ over the simplex $\{(x_v)_{v \in H} : x_v \geq 0, \sum_v
x_v=1\}$, which necessarily exists as the simplex is compact.
\end{defn}
We now relate $f(n,k,d)$ to the above definitions.
\begin{lem}
	\label{l21}
	For integers $n \geq d \geq 1$ and $k \geq 2^d$,
	\[d! \lambda(H(k,d)) \binom{n}{d} \leq f(n,k,d) \leq \lambda(H(k,d)) n^d.\]
	Equality holds in the upper bound if and only if $P_H$ is maximized at a point where all coordinates are in $\frac{1}{n}\setz$.
\end{lem}
\begin{proof}
	For the upper bound, suppose $M$ is a $k \times n$ binary matrix with 
	$f(n,k,d)$ shattered $k \times d$ submatrices. 
	For each $v\in \set{0,1}^k$, let $m_v$ denote the number of
	columns of $M$ which are equal to $v$ and define a weight
	$x_v=m_v/n$. These weights are clearly nonnegative and their sum
	is $1$. We claim that \[P_{H(k,d)}((x_v)_{v \in \set{0,1}^k})=f(n,k,d)/n^d,\] which will show the bound and the equality case by the definition of the Lagrangian.  Indeed, every edge $\set{v_1, v_2, \ldots ,v_d}$ of 
	$H(k,d)$
	contributes to the left-hand side exactly $m_{v_1}m_{v_2}\cdots m_{v_d}/n^d$, which is precisely the number of (shattered) 
	$k \times d$ submatrices 
	of $M$ with columns $v_1, v_2, \ldots ,v_d$ (in any order), divided by $n^d$. Summing over all edges of $H(k,d)$ yields the desired upper bound.
	
	For the lower bound, suppose $\bm{x}=(x_v)_{v \in \set{0,1}^k}$ is such that $P_{H(k,d)}(\bm{x}) = \lambda(H(k,d))$. Let $M$ be a $k \times n$ random binary matrix obtained by independently picking
		each column to be $v \in \set{0,1}^k$ with probability $x_v$. Given any $k \times d$ submatrix of $M$, the probability that it is shattered is $d! \lambda(H(k,d))$, since for every $\set{v_1, v_2, \ldots ,v_d} \in E(H(k,d))$ the probability that the matrix has columns $v_1, v_2, \ldots, v_d$ in some order is exactly $d! x_{v_1}x_{v_2} \cdots x_{v_d}$. Thus, by linearity of expectation, the expected number of shattered $d \times k$ submatrices of $M$ is $d! \lambda(H(k,d)) \binom{n}{d}$. 
\end{proof}

By fixing $k$ and $d$, and letting $n$ grow to infinity, we obtain the following:
\begin{cor} \label{cor:ckd}
For integers $d \geq 1$ and $k \geq 2^d$, we have $f(n,k,d) = (1 + o(1))\lambda(H(k,d)) n^d$.
\end{cor}
In what follows, we let $c(k,d) = d! \lambda(H(k,d))$.

\subsection{Relating different choices of \texorpdfstring{\boldmath $(k,d)$}{(k, d)}}
We start by remarking that since $f(n,k,d)$ is weakly increasing in $k$, the quantity $c(k,d)$ must also be weakly increasing in $k$. Another relation comes from the following simple property of the function $f(n,k,d)$:
\begin{lem}
	\label{l24}
 For integers $d \geq 1$, $k \geq 2^{d+1}$, and $n \geq d+1$, we have
	\[
	f(n, k, d+1) \leq \frac{n}{d+1} f(n-1, \lfloor k/2 \rfloor, d).
	\]
\end{lem}
\begin{proof}
	Let $M$ be a $k \times n$ binary matrix with $f(n,k,d+1)$ shattered
	$k \times (d+1)$ submatrices. Let $v$ be a fixed column of $M$;
	without loss of generality assume that the number of zeros
	it contains, $k'$, is at most $\lfloor k/2 \rfloor$. Let $M'$ be the 
	submatrix of $M$ consisting of all $k'$ rows of $M$ in which $v$
	has a zero, and all columns besides $v$. Note that for every 
	$k \times (d+1)$ shattered submatrix of $M$ that contains $v$, the
	corresponding $k' \times d$ submatrix of $M'$ must be shattered.
	Therefore, the number of shattered $k \times (d+1)$ submatrices of
	$M$ that contain $v$ is at most $f(n-1,\lfloor k/2 \rfloor, d)$.
	Summing over all columns $v$ of $M$, each shattered $k \times (d+1)$
	submatrix
	is counted $d+1$ times, implying the desired result.
\end{proof}
Applying \cref{cor:ckd} implies the following:
\begin{cor} \label{cor:2ckd}
For integers $d \geq 1$ and $k \geq 2^{d+1}$, we have $c(k, d+1) \leq c(\floor{k/2}, d)$.
\end{cor}
One helpful way to conceptualize this bound is to define, for every positive integer $d$, the weakly increasing step function $\gamma_d \colon [1, \infty) \to \setr$ given by $\gamma_d(b) = c(\floor{2^d b}, d)$. Then, \cref{cor:2ckd} rewrites as $\gamma_{d+1}(b) \leq \gamma_d(b)$. In particular, by the monotone convergence theorem, the limit $\gamma_\infty(b) \coloneq \lim_{d \to\infty} \gamma_d(b)$ exists. An illustration of this situation is shown in \cref{fig:1}.

\begin{figure}[tbp]
\centering
\begin{tikzpicture}[scale=5,null/.style={inner sep=0pt},ce/.style={inner sep=1pt,circle,draw,fill},oe/.style={ce,fill=white}]
\draw[->] (0.8,0)--(0.8,1.1);
\draw[->] (0.8,0)--(2.5,0) node[null,label={right:$b$}]{};

\draw[dotted] (0.8,1)--(1,1) (1,0)--(1,1) (1.5,0)--(1.5,0.75) (2,0)--(2,0.88) (1.25,0)--(1.25,0.38) (1.75,0)--(1.75,0.65) (2.25,0)--(2.25,0.8);

\draw (0.8,1) node[null,label={left:$1$}]{} +(-0.5pt,0)--+(0.5pt,0);
\draw (1,0) node[null,label={below:$1$}]{} +(0,-0.5pt)--+(0,0.5pt);
\draw (1.25,0) node[null,label={below:$\frac{9}{8}$}]{} +(0,-0.5pt)--+(0,0.5pt);
\draw (1.5,0) node[null,label={below:$\frac{5}{4}$}]{} +(0,-0.5pt)--+(0,0.5pt);
\draw (1.75,0) node[null,label={below:$\frac{11}{8}$}]{} +(0,-0.5pt)--+(0,0.5pt);
\draw (2,0) node[null,label={below:$\frac{3}{2}$}]{} +(0,-0.5pt)--+(0,0.5pt);
\draw (2.25,0) node[null,label={below:$\frac{13}{8}$}]{} +(0,-0.5pt)--+(0,0.5pt);

\draw[thick,red!80!black] (1,1)node[ce]{}--(2.4,1) node[null,label={right:$\gamma_1$}]{};
\draw[thick,green!60!black] (1,0.5)node[ce]{}--+(0.5,0)node[oe]{} (1.5,0.75)node[ce]{}--+(0.5,0)node[oe]{} (2,0.88)node[ce]{}--+(0.4,0) node[null,label={right:$\gamma_2$}]{};
\draw[thick,blue] (1,0.3)node[ce]{}--+(0.25,0)node[oe]{} (1.25, 0.38)node[ce]{}--+(0.25,0)node[oe]{} (1.5,0.57)node[ce]{}--+(0.25,0)node[oe]{} (1.75,0.65)node[ce]{}--+(0.25,0)node[oe]{} (2,0.72)node[ce]{}--+(0.25,0)node[oe]{} (2.25,0.8)node[ce]{}--+(0.15,0) node[null,label={right:$\gamma_3$}]{};
\draw[thick,Purple] (1,0.1)node[ce]{}..controls(1.2,0.25) and (1.3,0.3)..(1.5,0.35)node[oe]{} (1.5,0.42)node[ce]{}..controls(1.7,0.55) and (2,0.65)..(2.4,0.72) node[null,label={right:$\gamma_\infty$}]{};
\end{tikzpicture}
\caption{An illustration of the relationship between $\gamma_1$, $\gamma_2$, $\gamma_3$, and $\gamma_\infty$. Aside from $\gamma_1$, which is easily seen to be constant at $1$, all values are for illustrative purposes only.}\label{fig:1}
\end{figure}
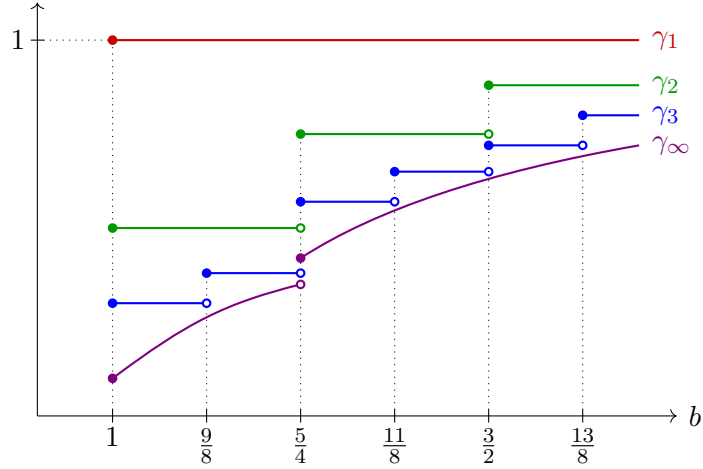

\begin{prop} \label{prop:rcon}
$\gamma_\infty$ is right-continuous.
\end{prop}
\begin{proof}
Take some $b \geq 1$ and $\eps > 0$. There must exist some $d$ where $\gamma_d(b) < \gamma_\infty(b) + \eps$. Since $\gamma_d$ is a step function, there exists some $\delta > 0$ such that $\gamma_d(b + \delta) = \gamma_d(b)$. Then for all $b < b' < b+\delta$, we have
\[\gamma_\infty(b') \leq \gamma_d(b') = \gamma_d(b) < \gamma_\infty(b) + \eps.\]
Since $\eps$ was arbitrary, $\gamma_\infty$ is right-continuous, as desired.
\end{proof}

The following lemma is now the appropriate generalization of \cref{t13a}.
\begin{lem} \label{lem:gammainfty}
Let $b \geq 1$ be a real number. If $d$, $k$, and $n$ are growing positive integers such that $d$ and $n-d$ tend to infinity, $k = (b + o(1))2^d$, and $k \geq b2^d$ always, then $f(n,k,d)=(\gamma_\infty(b)+o(1)) \binom{n}{d}$.
\end{lem}
\begin{proof}
To show the lower bound, note that by \cref{l21}, we have
\[f(n, k, d) > c(k,d) \binom{n}{d} \geq c(\floor{b2^d}, d) \binom{n}{d} = \gamma_d(b) \binom{n}{d} \geq \gamma_\infty(b) \binom{n}{d}.\]
To show the upper bound, we
choose an integer $0 \leq r < d$ so that 
$1 \ll (d-r)^2 \ll n-d$.
Applying \cref{l24} $r$ times and then \cref{l21}, we get
\begin{align*}
f(n,k,d) &\leq \frac{n(n-1) \cdots (n-r+1)}{d(d-1)
	\cdots (d-r+1)} f(n-r,\floor{k/2^r},d-r) \\ &\leq c(\floor{k/2^r}, d-r) \frac{n(n-1) \cdots (n-r+1)(n-r)^{d-r}}{d!}.
\end{align*}
Since $(d-r)^2 \ll n-d \leq n-r$, we find that
\begin{align*}
\binom{n}{d} \bigg/ \frac{n(n-1) \cdots (n-r+1)(n-r)^{d-r}}{d!}
&= \frac{(n-r)(n-r-1) \cdots (n-d+1)}{(n-r)^{d-r}} \\
&= \paren*{1 - \frac{1}{n-r}}\paren*{1 - \frac{2}{n-r}} \cdots \paren*{1 - \frac{d-r-1}{n-r}} \\
&= 1 + o(1),
\end{align*}
so it suffices to show that $c(\floor{k/2^r}, d-r) = \gamma_{d-r}(k/2^d)$ is bounded above by $\gamma_\infty(b) + o(1)$. Indeed, for every $\eps > 0$, since $k/2^d$ is eventually less than $b + \eps$, we eventually have $\gamma_{d-r}(k/2^d) \leq \gamma_{d-r}(b + \eps) = \gamma_\infty(b + \eps) + o(1)$. The result follows from \cref{prop:rcon}.
\end{proof}

\begin{rmk}
The assumption that $n-d$ tends to infinity is necessary for \cref{lem:gammainfty} to hold. Indeed, we trivially have $f(n,2^n,n)=1\;(=\binom{n}{n})$, and it is also easy to see that
$f(n,2^{n-1},n-1)=n \;(=\binom{n}{n-1})$ by taking the
collection of all even-sized subsets of $[n]$. On the other hand, $\gamma_\infty(1) < 1$.\footnote{We will determine the exact value of $\gamma_\infty(1)$ later, but an easy way to see this now is to use the fact that \emph{every} $d$-uniform hypergraph with $d \geq 2$ has a Lagrangian strictly less than $1/d!$.} In fact, for any fixed $s$ it can be shown that
$f(n,2^{n-s},n-s) \geq C_s \binom{n}{n-s}$ for some constant $C_s >\gamma_\infty(1)$.
\end{rmk}

\section{Shattering Pairs} \label{sec:pairs}
To begin, we recall a result of Kleitman and Spencer \cite{KS}. For a positive integer $k \geq 4$, call a collection $\KK\subseteq 2^{[k]}$ \emph{(qualitatively) pairwise independent} if for every two distinct
$A,B \in \KK$, all four intersections $A  \cap B$, $A \cap \bar B$,
$\bar A \cap B$, and $\bar A \cap \bar B$ are nonempty, 
where $\bar A=[k]\setminus A$ and $\bar B=[k]\setminus B$. We then have the following:
\begin{lem}[\cite{KS}]
	\label{l23}
	The maximum possible size of a  pairwise independent
	collection of subsets of $[k]$ is 
	\[\binom{k-1}{\lfloor k/2 \rfloor-1}.\]
\end{lem}
Given this result, the exact value of $f(n,k,2)$ follows quite quickly:

\begin{prop}
	\label{p31}
	For $k \geq 4$,
	\[f(n,k,2) = t\left(n, \binom{k-1}{\lfloor k/2 \rfloor - 1}\right),\]
	where \[
	t(n,r)=\sum_{0 \leq i<j \leq r-1} \left\lfloor \frac{n+i}{r} \right\rfloor 
	\left\lfloor \frac{n+j}{r} \right\rfloor
	\]
	is the number of edges of the Tur\'an graph $T(n,r)$, defined to be the
	complete $r$-partite graph with $n$ vertices and $r$ vertex classes with cardinalities that
	are as close as possible. In particular,
	\[c(k,2) = 1 - \frac{1}{\binom{k-1}{\lfloor k/2 \rfloor - 1}}.\]
\end{prop}
\begin{proof}
	We use the binary matrix interpretation developed in \cref{sec:general}. Given a $k\times n$ binary matrix $M$, construct a graph $G$ on the columns of $M$ by placing an edge between any two columns such that the corresponding $k \times 2$ submatrix is shattered. Thus, the number of shattered pairs is precisely the number of edges of $G$.
	
	By associating subsets of $[k]$ with their indicator vectors, \cref{l23} implies that the clique number of $G$ is at most $w \coloneq \binom{k-1}{\lfloor k/2 \rfloor - 1}$. Therefore, $G$ above is $K_{w+1}$-free, and so has at most 
	$t(n,w)$ edges by Tur\'an's Theorem. On the other hand, we can make equality hold by taking some pairwise independent $\KK \in 2^{[k]}$ with $\abs{\KK} = w$, considering the indicator vectors of the elements of $\KK$, and constructing a $k \times n$ 
	matrix in which each of these
	$w$ vectors appears either $\lfloor n/w \rfloor$ or $\lceil n/w
	\rceil$ times.
\end{proof}

\section{The Case \texorpdfstring{$k = 2^d$}{k = 2\textasciicircum d}} \label{sec:k2d}
We now state the key lemma for understanding the case $k = 2^d$, which was first proven in \cite{DM18}:
\begin{lem}[\cite{DM18}] \label{lem:2dlemma}
For $d$ a positive integer, we have 
\[c(2^d, d) = \frac{(2^d-2)(2^d-4) \cdots (2^d-2^{d-1})}{(2^d-1)^{d-1}} \eqcolon c_d.\]
Moreover, the Lagrangian polynomial of $H(2^d, d)$ attains its maximal value at a point with all coordinates in $\frac{1}{2^d-1} \setz$.
\end{lem}
Given this result, \cref{t12} follows by applying \cref{l21}. It also implies that $\gamma_\infty(1) = \lim_{d \to \infty} c_d = c$, so \cref{t13a} follows from \cref{lem:gammainfty}. \cref{t13} is a special case of \cref{t13a}.

For completeness, in this section we give two related proofs of the upper bound $c(2^d, d) \leq c_d$, the first of which works directly with the Lagrangian and is essentially the same as \cite{DM18}, and the second of which uses a result of Erd\H{o}s concerning degree majorization. The lower bound and its associated equality case will also be shown in \cref{sec:kmedium}, where it will follow easily from some more general techniques.

In both proofs of the upper bound, we will need the following useful lemma, which appears in both \cite{DM18} and \cite{Al}.
\begin{lem}
\label{l25}
Let $d \geq 2$ and let $(p_i: 1 \leq i \leq 2^d-1)$ be
an arbitrary
probability distribution on a set of size $2^d-1$. Then
\[
\sum_i p_i(1-p_i)^{d-1}
\leq \left(\frac{2^d-2}{2^d-1}\right)^{d-1}.
\]
\end{lem}

\subsection{Proof of upper bound using Lagrangians} \label{sec:proof1}
We apply induction on $d$.
For $d=1$, $H(2,1)$ is a $1$-uniform hypergraph  with two
edges (singletons) corresponding to the vectors $01$ and $10$.
It is clear that $\lambda(H(2,1))=1=c_1/1!$, as needed.

Assuming the result holds for $d-1$, we prove it for $d$, where $d \geq 2$.
Let $D=2^d$ and 
let $P$ be the Lagrangian polynomial of
$H(2^d,d)=H(D,d)$. Suppose it attains its maximum at the point
$\bm{x}=(x_v)_{v \in \set{0,1}^D}$, where the vector $\bm{x}$ has
a support $S$ of minimum possible size among all vectors maximizing
$P$. By a well-known property of Lagrangians, every pair $u,v$
of distinct vertices in $S$ is contained in an edge of
$H(D,d)$. Indeed, if not, then when fixing the values of all $x_w$ where $w \neq u,v$, the function $P(\bm{x})$ is a linear function
of $x_u$ and $x_v$, so its maximum subject to the constraints $x_u,x_v \geq 0$ and $x_u+x_v=a$ for some $a$ is attained 
at a point where either $x_u=0$ or $x_v=0$, which must have a smaller support.  We may thus assume that every pair of
vectors $u,v$ with $x_u,x_v>0$ is contained in an edge of
$H(D,d)$.

One easy consequence of this is that every $v \in S$ contains exactly $D/2$ zeros and $D/2$ ones. Moreover, since the vectors
\[\set{((-1)^{v_1}, (-1)^{v_2},
\ldots ,(-1)^{v_D}) : (v_1,v_2,\ldots,v_D) \in S}\] are mutually orthogonal and are additionally orthogonal to the all-$1$s vector, we find that $\abs{S} \leq D-1$.

Fix some $v \in S$ and let $I$ denote the set of indices of its $2^{d-1}$ 
$0$-coordinates. Then the link of $v$ in $H(D,d)$ is a
$(d-1)$-uniform hypergraph in which every edge is a set of
$d-1$ vectors whose $I$-coordinates 
form a $2^{d-1} \times (d-1)$ shattered matrix; in other words, after identifying 
vertices that have the same $I$-coordinates, this hypergraph becomes $H(2^{d-1}, d-1)$. By adding weights of
identified vertices, it follows that the contribution
of all edges containing $v$ to the sum in the expression of
$P(\bm{x})$ is at most 
$x_v \cdot \lambda(H(2^{d-1},d-1))(1-x_v)^{d-1}$. By the induction
hypothesis, $\lambda(H(2^{d-1},d-1)) \leq \frac{c_{d-1}}{(d-1)!}$.
Summing over all $v$ we get every term $d$ times and hence
\[
P(\bm{x}) \leq \frac{1}{d}  \cdot
\frac{c_{d-1}}{(d-1)!} \sum_v x_v (1-x_v)^{d-1} \leq \frac{c_{d-1}}{d!} 
\left(\frac{2^d-2}{2^d-1}\right)^{d-1},
\]
where we have used \cref{l25}. This last quantity can be easily checked to be $c_d/d!$. (See Equation (2) in \cite{Al} for a combinatorial explanation of this fact.)

\subsection{Proof of upper bound using degree majorization}
For this proof, we will use a result of Erd\H{o}s, which was originally
used to provide a proof of Tur\'an's Theorem. Suppose $G$ and $H$
be two graphs on $n$ vertices and let $d_1 \geq d_2 \geq \ldots \geq
d_n$ be the degrees of the vertices of $G$ and
$f_1 \geq f_2 \ldots \geq f_n$ be the degrees of the vertices of
$H$. Say $G$ is \emph{degree-majorized} by $H$ if
$d_i \leq f_i$ for all $i$.
\begin{lem}[\cite{Er}]
\label{l22}
If $G$ is a graph on $n$ vertices that contains no clique of
size $w+1$ then it is degree-majorized by some complete $w$-partite 
graph on $n$ vertices.
\end{lem}

We now prove that $f(n, 2^d, d) \leq c_d n^d/d!$ by induction on $n$. As the case $d=1$ is trivial, assume that the result holds for $d-1$ with $d \geq 2$; we will prove it for $d$.

Let $M$ be a $2^d \times n$ binary matrix with the
maximum possible number $f(n,2^d,d)$ of shattered $2^d \times d$
submatrices. We may assume that
every column of $M$ contains exactly $2^{d-1}$ zeros and
exactly $2^{d-1}$ ones. Construct the graph $G$ on the 
columns of $M$, where two columns are joined by an edge if there are
exactly $2^{d-2}$ coordinates in which both are $0$ (which forces
exactly $2^{d-2}$ coordinates in which one column is $a$ and the other is $a'$ for all $a,a' \in \set{0,1}$). Note that for
any fixed column $v$, the other columns in 
any shattered $2^d \times d$ submatrix of $M$
that contains $v$ must be connected to $v$ in the graph $G$.
Moreover, after deleting $v$ and restricting to the
$2^{d-1}$ rows in which $v$ has a zero, we get a shattered 
$2^{d-1} \times (d-1)$ matrix. By the induction hypothesis this
implies that if the degree of $v$ in $G$ is $d_v$, then the number
of shattered $2^d \times d$ submatrices containing it is at most $c_{d-1} d_v^{d-1}/(d-1)!$.

By the same orthogonality argument as in \cref{sec:proof1}, the
largest clique of $G$ is of size at most $w=2^d-1$. Therefore,
by \cref{l22}, there is a complete $w$-partite graph $H$ on
$n$ vertices so that $G$ is degree majorized by $H$. If the sizes
of the vertex classes of this graph are $n_1,n_2, \ldots ,n_w$, it follows that the vertices of $G$ can be partitioned into subsets of
sizes $n_1, n_2, \ldots ,n_w$,
where each of the $n_i$ vertices in the $i$th subset
has degree at most
$n-n_i$. Summing over all columns $v$ we conclude
\begin{multline*}
f(n,2^d,d) \leq \frac{1}{d} \sum_i n_i \frac{c_{d-1}}{(d-1)!}
(n-n_i)^{d-1} \\ = \frac{c_{d-1}n^d}{d!} \sum_i 
\frac{n_i}{n} \paren*{1-\frac{n_i}{n}}^{d-1} \leq \frac{c_{d-1} n^d}{d!} \left(\frac{2^d-2}{2^d-1}\right)^{d-1} 
n^d=\frac{c_dn^d}{d!},
\end{multline*}
where the second inequality uses \cref{l25} on $(n_i/n)_{i \in [w]}$. This concludes the proof.

\section{General \texorpdfstring{$k$}{k}} \label{sec:kg2d}
As mentioned in the introduction, the behavior of 
$f(n,k,d)$ and $c(k,d)$ when $k>2^d$ is less understood than the case $k = 2^d$. Although the discussion in \cref{sec:pairs} settles the case $d=2$, we do not have any upper bounds better than using \cref{l24} to reduce to either the case $d=2$ or the case $k = 2^d$. In this section, we detail lower bounds on $c(k,d)$ in two regimes: when $2^d \leq k \leq 2^{d+1}$ and when $k \gg 2^d$. The former case also contains the proof of the lower bound of \cref{l25}. 
\subsection{\texorpdfstring{\boldmath $k \leq 2^{d+1}$}{k ≤ 2\textasciicircum (d+1)}} \label{sec:kmedium}
All of our constructions will use the following lemma:
\begin{lem} \label{lem:linalg}
Let $d \geq 1$ and $k \geq 2^d$ be integers. Let $V$ be a $d'$-dimensional $\setf_2$-vector space and let $S$ be a subset of $V$ of size $k$. If $p$ is the probability that a uniformly random linear map $V \to \setf_2^d$ is surjective when restricted to $S$, then there exists a $k \times (2^{d'}-1)$ matrix with $2^{dd'}p/d!$ shattered $k \times d$ submatrices. In particular,
\[c(k, d) \geq \paren*{\frac{2^{d'}}{2^{d'}-1}}^d p.\]
\end{lem}
\begin{proof}
The matrix we will construct has its rows indexed by the elements of $S$ and its columns indexed by the nonzero elements of the dual space $V^*$ of $V$. Given $v \in S$ and nonzero $u \in V^*$, the corresponding entry is $\gen{u, v}$.

Every linear map $\varphi\colon V \to \setf_2^d$ can be uniquely expressed as a tuple $(u_1, u_2,\ldots,u_d)$ of dual vectors. If one of the $u_i$ is zero or if $u_i = u_j$ for some $i \neq j$, then $\varphi$ cannot be surjective. Otherwise, $\varphi(S) = \setf_2^d$ if and only if the submatrix of $M$ with columns given by $u_1,u_2,\ldots,u_d$ is shattered. As a result, the $2^{dd'}p$ linear maps $V \to \setf_2^d$ that are surjective on $S$ each correspond to an ordered $d$-tuple of distinct columns of $M$ that determine a shattered submatrix. As each shattered submatrix corresponds to $d!$ such $d$-tuples, there are $2^{dd'}p/d!$ shattered submatrices, as desired.
\end{proof}

\begin{proof}[Proof of lower bound and equality case of \cref{lem:2dlemma}]
We apply \cref{lem:linalg} with $V = S = \setf_2^d$. In this case, $p$ is the probability that a random $d \times d$ $\setf_2$-matrix is invertible, which is exactly
\[\frac{(2^d-2^0)(2^d-2^1) \cdots (2^d-2^{d-1})}{2^{d^2}} = \frac{(2^d-1)^d c_d}{2^{d^2}}.\]
Thus we get $c(2^d, d) \geq c_d$. The fact that a $2^d \times (2^d-1)$ matrix exists with $c_d(2^d-1)^d/d!$ shattered submatrices implies, by \cref{l21}, that the corresponding Lagrangian polynomial attains a value of $c_d/d!$ at a point on $\frac{1}{2^d-1} \setz$.
\end{proof}

\begin{lem} \label{lem:codim}
For integers $d \geq 1$ and $0 \leq r \leq d$, we have $c((2-2^{-r})2^d, d) \geq (2-2^{-r})c_{d+1}$.
\end{lem}
\begin{proof}
We apply \cref{lem:linalg} with $V$ a $(d+1)$-dimensional space and $S = V \setminus W$, where $W$ is a $(d-r)$-dimensional subspace of $V$. Consider a linear map $\varphi\colon V \to \setf_2^d$ and the induced map $\bar\varphi \colon V/W \to \setf_2^d/\varphi(W)$. Since $S$ is the union of translates of $W$, the map $\varphi$ is surjective on $S$ if and only if $\bar\varphi$ is surjective on the nonzero elements of $V/W$. If $\varphi$ is not injective on $W$ this is impossible, since $\abs{V/W} \leq \abs{\setf_2^d/\varphi(W)}$. If $\varphi$ is injective on $W$, then this occurs if and only if $\bar\varphi$ is surjective, since the fact that $\dim(V/W) = \dim(\setf_2^d/\varphi(W)) + 1$ implies that every element in the codomain has a preimage of size exactly $2$.

If $\varphi$ is chosen uniformly at random, the probability that $\varphi \rvert_W$ is injective is
\[\frac{(2^d-2^0)(2^d-2^1) \cdots (2^d-2^{d-r-1})}{2^{d(d-r)}}.\]
If we fix $\varphi \rvert_W$, then $\bar\varphi$ is uniformly distributed, so the probability that it is surjective is the probability that a random $r \times (r+1)$ $\setf_2$-matrix has rank $r$, which is
\[\frac{(2^{r+1}-2^0)(2^{r+1}-2^1) \cdots (2^{r+1}-2^{r-1})}{2^{r(r+1)}}.\]
We conclude that, with $p$ defined as in \cref{lem:linalg},
\begin{align*}
p &= \frac{(2^d-2^0)(2^d-2^1) \cdots (2^d-2^{d-r-1})}{2^{d(d-r)}} \cdot \frac{(2^{r+1}-2^0)(2^{r+1}-2^1) \cdots (2^{r+1}-2^{r-1})}{2^{r(r+1)}} \\
&= \frac{(2^d-2^0)(2^d-2^1) \cdots (2^d-2^{d-r-1})}{2^{d(d-r)}} \cdot \frac{(2^{d}-2^{d-r-1})(2^{d}-2^{d-r-2}) \cdots (2^{d}-2^{d-2})}{2^{dr}} \\
&= \frac{(2^d-2^0)(2^d-2^1) \cdots (2^d-2^{d-1})}{2^{d^2}} \cdot \frac{2^d-2^{d-r-1}}{2^d - 2^{d-1}} \\
&= (2 - 2^{-r})\frac{(2^d-2^0)(2^d-2^1) \cdots (2^d-2^{d-1})}{2^{d^2}} \\
&= (2 - 2^{-r}) \paren*{\frac{2^{d+1}-1}{2^{d+1}}}^d c_{d+1}.
\end{align*}
The result follows.
\end{proof}
It seems plausible to conjecture that the lower bounds in \cref{lem:2dlemma,lem:codim} are the best possible; specifically, that for every $d > 2$,
\[c(k, d) = \begin{cases}
c_d & 2^d \leq k < \frac{3}{2} \cdot 2^d \\
\frac{3}{2} c_{d+1} & \frac{3}{2} \cdot 2^d \leq k < \frac{7}{4} \cdot 2^d \\
\frac{7}{4} c_{d+1} & \frac{7}{4} \cdot 2^d \leq k < \frac{15}{8} \cdot 2^d \\
\mkern3mu\vdots & \\
(2 - 2^{-d})c_{d+1} & k = 2^{d+1}-1. \\
\end{cases}\]
A heuristic computer search was unable to disprove this conjecture in the cases $(k, d) = (9, 3), (10, 3)$. However, it should be noted that this statement for $d=2$ is in fact false by \cref{p31}. A safer conjecture to make may be that these bounds are optimal in the limit $d \to \infty$, i.e.\ for $1 \leq b < 2$, we have $\gamma_\infty(b) = (2 - 2^{\ceil{\log_2(2-b)}}) c$. This function is plotted in \cref{fig:2}.

\begin{figure}[tbp]
\centering
\begin{tikzpicture}[scale=5,null/.style={inner sep=0pt},ce/.style={inner sep=1pt,circle,draw,fill},oe/.style={ce,fill=white}]
\draw[->] (-0.1,-0.1)--(-0.1,1.1);
\draw[->] (-0.1,-0.1)--(1.1,-0.1) node[null,label={right:$b$}]{};

\draw (-0.1,0) node[null,label={left:$c$}]{} +(-0.5pt,0)--+(0.5pt,0);
\draw (-0.1,0.5) node[null,label={left:$\frac{3}{2}c$}]{} +(-0.5pt,0)--+(0.5pt,0);
\draw (-0.1,0.75) node[null,label={left:$\frac{7}{4}c$}]{} +(-0.5pt,0)--+(0.5pt,0);
\draw (-0.1,7/8) node[null,label={left:$\frac{15}{8}c$}]{} +(-0.5pt,0)--+(0.5pt,0);
\draw (-0.1,1) node[null,label={left:$2c$}]{} +(-0.5pt,0)--+(0.5pt,0);
\draw (0,-0.1) node[null,label={below:$1$}]{} +(0,-0.5pt)--+(0,0.5pt);
\draw (0.5,-0.1) node[null,label={below:$\frac{3}{2}$}]{} +(0,-0.5pt)--+(0,0.5pt);
\draw (0.75,-0.1) node[null,label={below:$\frac{7}{4}$}]{} +(0,-0.5pt)--+(0,0.5pt);
\draw (7/8,-0.1) node[null,label={below:$\frac{15}{8}$}]{} +(0,-0.5pt)--+(0,0.5pt);
\draw (1,-0.1) node[null,label={below:$2$}]{} +(0,-0.5pt)--+(0,0.5pt);

\draw[dotted] (-0.1,0)--(0,0)--(0,-0.1) (-0.1,0.5)--(0.5,0.5)--(0.5,-0.1) (-0.1,0.75)--(0.75,0.75)--(0.75,-0.1) (-0.1,7/8)--(7/8,7/8)--(7/8,-0.1) (-0.1,1)--(1,1)--(1,-0.1) (0,0)--(1,1) (0.5,0)--(1,1);

\draw[thick] (0,0)node[ce]{}--(0.5,0)node[oe]{} (0.5,0.5)node[ce]{}--(0.75,0.5)node[oe]{} (0.75,0.75)node[ce]{}--(7/8,3/4)node[oe]{};
\begin{scope}[shift={(7/8,7/8)},scale=1/8]
\begin{scope}[shift={(7/8,7/8)},scale=1/8]
\draw[thick] (0.75,0.75)node[ce]{}--(7/8,3/4)node[oe]{} (0.5,0.5)node[ce]{}--(0.75,0.5)node[oe]{} (0,0)node[ce]{}--(0.5,0)node[oe]{} (1,1) node[oe]{};
\end{scope}
\draw[thick] (0,0)node[ce]{}--(0.5,0)node[oe]{} (0.5,0.5)node[ce]{}--(0.75,0.5)node[oe]{} (0.75,0.75)node[ce]{}--(7/8,3/4)node[oe]{};
\end{scope}
\end{tikzpicture}
\caption{The function $b \mapsto (2 - 2^{\ceil{\log_2(2-b)}}) c$, which is both the best known lower bound and conjectured form for $\gamma_\infty(b)$ in the range $[1, 2)$.} \label{fig:2}
\end{figure}
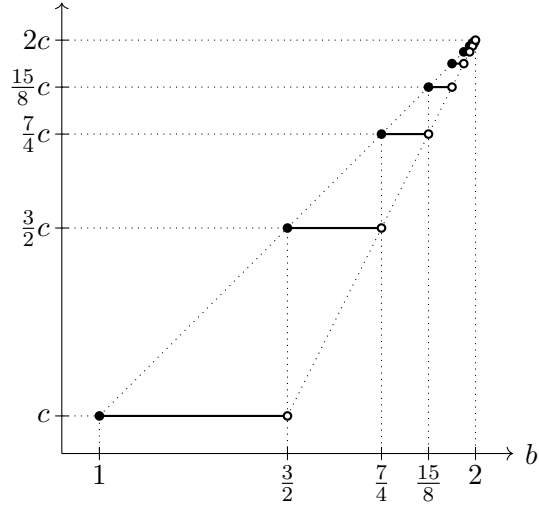

\subsection{Very large \texorpdfstring{\boldmath$k$}{k}} \label{sec:largek}
We start with the following observation:
\begin{lem}
Let $d \geq 1$, $k_1, k_2 \geq 2^d$, and $n_1, n_2 \geq d$ be integers. Then
\[(n_1n_2)^d - d!f(n_1n_2, k_1+k_2, d) \leq (n_1^d - d!f(n_1,k_1,d))(n_2^d - d!f(n_2,k_2,d)).\]
\end{lem}
\begin{proof}
Given a $k \times n$ matrix $M$, let $X_M$ be a random $k \times d$ matrix obtained by choosing $d$ columns of $M$ independently and uniformly at random. Note that if $M$ has $m$ shattered $k \times d$ submatrices, the probability that $X_M$ is shattered is precisely $d! m/n^d$.

Suppose $M_1$ and $M_2$ are $k_1 \times n_1$ and $k_2 \times n_2$ matrices with $f(n_1,k_1,d)$ and $f(n_2,k_2,d)$ shattered $k \times d$ submatrices, respectively. Then, let $M$ be the $(k_1 + k_2) \times (n_1n_2)$ matrix whose columns are the concatenations of any column of $M_1$ and any column of $M_2$. It is evident that $X_M$ consists of $X_{M_1}$ stacked on top of an independent $X_{M_2}$, so
\[\setp[X_M \text{ not shattered}] \leq \setp[X_{M_1} \text{ not shattered}] \cdot \setp[X_{M_2} \text{ not shattered}].\]
The result follows after some algebra.
\end{proof}
Taking the limit $n \to \infty$, we conclude that $(1-c(k_1+k_2, d)) \leq (1-c(k_1, d)) (1-c(k_2, d))$, or equivalently, $(1-\gamma_d(b_1+b_2)) \leq (1-\gamma_d(b_1))(1-\gamma_d(b_2))$. Taking the limit $d \to \infty$ shows that $(1-\gamma_\infty(b_1+b_2)) \leq (1-\gamma_\infty(b_1))(1-\gamma_\infty(b_2))$ as well. Applying Fekete's lemma to $\log(1-\gamma_d(-))$ for possibly infinite $d$, we find that either $1-\gamma_d(b) = \beta_d^{-(1+o(1))b}$ for some finite $\beta_d = \sup_{b \geq 1}(1-\gamma_d(b))^{-1/b}$ or $1-\gamma_d(b)$ decays superexponentially. \cref{p31} not only tells us that $\beta_2 = 16$, but also rules out superexponential decay for all $d \geq 2$ as $1 - \gamma_d(b) \geq 1-\gamma_2(b)$. Note that $\gamma_2(b) \geq \gamma_3(b) \geq \cdots \geq \gamma_\infty(b)$ implies that $\beta_2 \geq \beta_3 \geq \cdots \geq \beta_\infty$.

The following simple observation ends up outperforming all other lower bounds considered in this paper when $b$ is larger than a constant times $d$.
\begin{prop} \label{prop:random}
$\beta_d \geq (1-2^{-d})^{-2^d} = e(1 + (\frac{1}{2} + o(1))2^{-d})$.
\end{prop}
\begin{proof}
Consider a uniformly random $k \times d$ binary matrix. The probability that a fixed element of $\set{0,1}^d$ does not appear as a row is $(1-2^{-d})^k$, so the matrix is shattered with probability at least $1 - 2^d(1-2^{-d})^k$. By picking uniformly random $k \times n$ binary matrices for large $n$ (or, equivalently, by plugging in a constant vector to the Lagrangian polynomial), we find that $c(k, d) \geq 1 - 2^d(1-2^{-d})^k$. The result follows.
\end{proof}

It is in fact possible to squeeze a bit more out of this idea by slightly optimizing the random process.
\begin{prop} \label{prop:balanced}
$\beta_d \geq \sup_{t \in \setr} {\paren{(\cosh t)^d - e^{dt}/2^d}^{-2^d} = e(1 + (\frac{d+1}{2} + o(1))2^{-d}})$.
\end{prop}
\begin{proof}
Let $\beta'_d = \sup_{t \in \setr} \paren{(\cosh t)^d - e^{dt}/2^d}^{-2^d}$. Suppose $k$ is even and choose a uniformly random $k \times d$ matrix subject to the condition that all columns have exactly $k/2$ zeros and $k/2$ ones (call a column \emph{balanced} if this is true and a matrix balanced if all its columns are balanced). It suffices to show that the probability $p$ that this matrix is not shattered is at most $(\beta'_d)^{-k/2^d}\exp(o_d(k))$, since then picking uniformly random balanced $k \times n$ matrices yields $c(k,d) \geq 1 - (\beta'_d)^{-k/2^d}\exp(o_d(k))$ for even $k$.

The number of balanced $k \times d$ matrices $\binom{k}{k/2}^d = 2^{dk+o_d(k)}$. Since toggling a column does not change whether it is balanced, to bound the number of balanced $k \times d$ matrices that are not shattered, it suffices to count balanced $k \times d$ matrices that lack an all-ones row, and then multiply by $2^d$.

Treating a $k \times d$ matrix as an $k$-tuple of its rows, the number of balanced $k \times d$ matrices without an all-ones rows is exactly the coefficient of $(x_1\cdots x_d)^{k/2}$ in the generating function $((1 + x_1) \cdots (1+x_d) - x_1 \cdots x_d)^k$. Thus the number of such matrices is at most
\[\frac{((1 + x_1) \cdots (1+x_d) - x_1\cdots x_d)^k}{(x_1\cdots x_d)^{k/2}}\]
for any choice of positive $x_1,\ldots,x_d$. Setting all the $x_i$ to be equal to $e^{2t}$, this is
\[\frac{(2^de^{dt}(\cosh t)^d  - e^{2dt})^k}{e^{dkt}} = 2^{dk}(\cosh(t)^d  - e^{dt}/2^d)^k.\]
Putting everything together, we find that
\[p \leq \frac{2^d2^{dk}(\cosh t)^d  - e^{dt}/2^d)^k}{2^{dk+o_d(k)}} = (\cosh(t)^d  - e^{dt}/2^d)^{k} \exp(o_d(k)).\]
By optimizing the choice of $t$, we get the desired bound.

We now compute the asymptotics of $\beta'_d$. Let $f(t) = (\cosh t)^d - e^{dt}/2^d$; by expanding out $(\cosh t)^d$, one can show that $f(t)$ is a positive linear combination of exponentials and is thus convex. Now, after computing
\[f'''(t) = d(3d-2)\sinh t(\cosh t)^{d-1} + d(d-1)(d-2) (\sinh t)^3(\cosh t)^{d-3} - d^3e^{dt}/2^d,\]
we find that for $\abs{t} \leq d^{-100}$, we have $\abs{f'''(t)} \leq 1$ for large enough $d$. Therefore, by Taylor's theorem, we find that for large $d$ and $\abs{t} \leq d^{-100}$, we have $\abs{f(t) - g(t)} = O(t^3)$ and $\abs{f'(t) - g'(t)} = O(t^2)$, where
\[g(t) = 1+dt^2/2-\frac{1-dt+d^2t^2/2}{2^d} = (1-2^{-d}) - d2^{-d} \cdot t + \frac{d + d^2/2^d}{2} \cdot t^2\]
is the second-degree Taylor polynomial of $f(t)$ at $t = 0$.

Computing $g'(t) = (d+d^2/2^d) t - d2^{-d}$, we find that if we define $t_\pm = 2^{-d} \pm 2^{-1.5d}$, both $g'(t_+)$ and $-g'(t_-)$ are $\Omega(2^{-1.5d})$. Thus, $f'(t_+)$ and $-f'(t_-)$ are also $\Omega(2^{-1.5d})$, so by convexity $f$ must be minimized in the interval $[t_-, t_+]$ for large $d$. Moreover, $g(t)$ is minimized at $t_0 = 2^{-d} + O(2^{-2d}) \in [t_-, t_+]$, and $g(t_0) = 1 - 2^{-d} - (d/2 + o(1)) 2^{-2d}$. Therefore the minimum of $f$, which is $(\beta'_d)^{-1/2^d}$, is $g(t_0) + O(2^{-3d}) = 1 - 2^{-d} - (d/2 + o(1)) 2^{-2d}$. It is now straightforward to compute
\[\log \beta'_d = -2^d \cdot \paren*{-2^d - \frac{d}{2} 2^{-2d} - \frac{1}{2} 2^{-2d}+ o(2^{-2d})} = 1 + \paren*{\frac{d+1}{2} + o(1)} 2^{-d}.\]
The result follows.
\end{proof}
\begin{rmk}
By the theory of large deviations in probability, this bound on $\beta_d$ is in fact the best possible for this probabilistic procedure.
\end{rmk}

\begin{rmk}
Although we have proved bounds on various $\beta_d$, it may not be the case that $\beta_\infty = \lim_{d \to\infty} \beta_d$. For instance, the functions $\min(2^{-x}, 3^{-x}), \min(2^{-x}, 3^{1-x}), \min(2^{-x}, 3^{2-x}), \ldots$ are each individually $3^{-(1+o(1))x}$, but their pointwise limit is exactly $2^{-x}$. The best bound we know for $\beta_\infty$ comes from using \cref{lem:codim} to conclude $c(2^{d+1},d) \geq (2 - 2^{-d}) c_{d+1}$, which implies that $\gamma_\infty(2) \geq 2c$ and thus $\beta_\infty \geq (1-2c)^{-1/2} \approx 1.539$.
\end{rmk}

\section{Concluding Remarks} \label{sec:final}
\subsection{Minimum shattering for fixed \texorpdfstring{\boldmath $d$}{d}} \label{sec:minshatter}
As mentioned in the introduction, the problem of determining 
$g(n,k,d)$, which is the \emph{minimum} possible number of 
subsets of size $d$ of $[n]$ which are shattered by
a family $\FF$ of $k$ distinct subsets of $[n]$, is much simpler
than that of determining $f(n,k,d)$. An explicit formula for
the value of $g(n,k,d)$ is somewhat complicated, we illustrate the 
way of computing it by describing the formula for some range of the
parameters. Writing $\binom{n}{{<}d}=\sum_{i=0}^{d-1}\binom{n}{i}$, 
let $r \in [d, n]$ and suppose that $k$ satisfies
\begin{multline*}
\binom{n}{{<}d}+\binom{r}{d} +\brac*{\binom{r}{d+1}-1}
+\brac*{\binom{r}{d+2}-1} + \cdots + \brac*{\binom{r}{r-1}-1} \leq k
\\
\leq 
\binom{n}{{<}d} +\binom{r}{d} +\binom{r}{d+1}
+\binom{r}{d+2} + \cdots +\binom{r}{r-1} + \binom{r}{r}.
\end{multline*}
We claim that in this range $g(n,k,d)=\binom{r}{d}$.
To prove the upper bound it suffices to to establish it for the upper limit of this range,
since $g(n,k,d)$ is clearly weakly increasing in $k$. Let $\FF$
be the family of all subsets of size at most $d-1$ of $[n]$
together with all subsets of $[r]$. Then $\abs{\FF}=k $ and the $d$-subsets of
$[n]$ it shatters are exactly all $d$-subsets of $[r]$. To prove
the lower bound it suffices to prove that any family $\FF \subseteq 2^{[n]}$ of size 
\[
\binom{n}{{<}d}+\binom{r}{d} +\brac*{\binom{r}{d+1}-1}
+\brac*{\binom{r}{d+2}-1} + \cdots + \brac*{\binom{r}{r-1}-1}
\] 
shatters at least $\binom{r}{d}$ subsets of size $d$ of
$[n]$. By the result of Pajor mentioned in the introduction,
$\FF$ shatters at least $\abs{\FF}$ subsets of $[n]$. Note that the
family of all shattered subsets forms a simplicial complex, namely,
it is closed under taking subsets. This complex contains
at most $\binom{n}{{<}d}$ subsets of size at most $d-1$. If it
contains a subset of size $r'$ for some $r' \geq r$, then it
contains at least $\binom{r'}{d}\geq \binom{r}{d}$ subsets of
size $d$, as needed. Similarly, if it contains at least
$\binom{r}{i}$ subsets of size $i$ for some $i \geq d$,
then
by the Kruskal-Katona Theorem it contains at least $\binom{r}{d}$
subsets of size $d$, as required. If none of these
conditions holds, then
\[
|\FF| \leq \binom{n}{{<}d}+\brac*{\binom{r}{d}-1} +\brac*{\binom{r}{d+1}-1}
+\brac*{\binom{r}{d+2}-1} + \cdots + \brac*{\binom{r}{r-1}-1}.
\]
which is smaller than the assumed size. This completes the 
proof of the claim providing an explicit formula for  
$g(n,k,d)$ in this range.

In general, the optimal construction comes from first putting all subsets of $[n]$ with size less than $d$ in $\FF$, and then adding the remaining subsets in lexicographic order, without regard to their size.

\subsection{Larger alphabets} \label{sec:alpha}
This problem, in the binary matrix formulation, naturally generalizes to an alphabet of size $v$. Most of the arguments in this paper generalize, with two main exceptions. First of all, we do not have an exact analogue of \cref{l23}, so the $d = 2$ case is significantly more mysterious. We note, however, that an asymptotic 
version of the analogue of \cref{l23} 
has been obtained by Gargano, K\"orner, and Vaccaro
\cite{GKV} using an elegant construction motivated by techniques from
information theory.

Second, unless $v$ is a prime or a prime power, constructions involving finite field linear algebra stop working. Nonetheless, it is still possible to salvage something. Letting $f_v(n,k,d)$ be the natural generalization of $f(n,k,d)$ to an alphabet of size $v$, we have the following:
\begin{prop} \label{prop:multv}
$f_{v_1v_2}(n_1n_2,k_1k_2, d) \geq d! f_{v_1}(n_1,k_1, d) f_{v_2}(n_2,k_2, d)$
\end{prop}
\begin{proof}
Consider matrices $M_1 \in [v_1]^{k_1 \times n_1}$ and $M_2 \in [v_2]^{k_2 \times n_2}$ with $f_{v_1}(n_1,k_1, d)$ and $f_{v_2}(n_2,k_2, d)$ shattered submatrices, respectively. Let $M \in ([v_1] \times [v_2])^{k_1k_2 \times n_1n_2}$ be such that for $i_1 \in [k_1]$, $i_2 \in [k_2]$, $j_1 \in [n_1]$, and $j_2 \in [n_2]$, we have
\[M_{(i_1,i_2),(j_1,j_2)} = ((M_1)_{i_1,j_1}, (M_2)_{i_2,j_2}).\]
One can check that if the $k_1 \times d$ submatrix of $M_1$ given by columns $j_1,\ldots,j_d$ and the $k_2 \times d$ submatrix of $M_2$ given by columns $j'_1,\ldots,j'_d$ are both shattered, the $k_1k_2 \times d$ submatrix of $M$ given by columns $(j_1,j'_1),\ldots,(j_d,j'_d)$ is shattered. This proves the desired bound, as there are $d!$ ways to combine a pair of shattered submatrices of $M_1$ and $M_2$.
\end{proof}

As a corollary, we find that, after defining $c_v(k,d)$ and $\gamma_{v,d}(s)$ to be the natural generalizations of $c(k,d)$ and $\gamma_d(s)$, we have $c_{v_1v_2}(k_1k_2, d) \geq c_{v_1}(k_1, d) c_{v_2}(k_2, d)$ and $\gamma_{v_1v_2,d}(s_1s_2) \geq \gamma_{v_1,d}(s_1)\gamma_{v_2,d}(s_2)$. In particular, $\gamma_{v,\infty}(1) > 0$ for all $v$, since we can write every $v$ as a product of prime powers.

An interesting phenomenon which occurs for $v \geq 2$ is that the best known bounds for $\lim_{d\to\infty}\beta_{v,d}$ and $\beta_{v,\infty}$ depend on the factorization of $v$. Completely random constructions (see \cref{prop:random}) yield $\lim_{d\to\infty}\beta_{v,d} \geq e$ unconditionally, while combining linear-algebraic constructions and \cref{prop:multv} yields
\[\lim_{d\to\infty}\beta_{v,d} \geq \beta_{v,\infty} \geq \frac{1}{1-\gamma_{v,\infty}(1)} \geq \frac{1}{1 - \prod_q \prod_{i=1}^\infty (1-q^{-i})},\]
where the product is over maximal prime powers $q$ that divide $v$. This is $v-1+o(1)$ for large prime power $v$, and is larger than $e$ for all prime powers $v \geq 4$, as well as for some $v$ that are not prime powers but products of large prime powers, the smallest of which is $v=35$. Moreover, if $v \equiv 2 \pmod{4}$, using the fact that $\gamma_{2,\infty}(2) \geq 2c$ yields
\[\beta_{v,\infty} \geq \frac{1}{\sqrt{1-\gamma_{v,\infty}(2)}} \geq \frac{1}{\sqrt{1 - 2\prod_q \prod_{i=1}^\infty (1-q^{-i})}},\]
which is always better as $\sqrt{1-2a} < 1-a$ for $a \in (0,1/2)$. However, in this case, since $(1-2c)^{-1/2} < e$, this bound is always less than $e$ and thus does not improve on the random construction for $\lim_{d\to\infty}\beta_{v,d}$.
\subsection{Application to covering arrays} \label{sec:cover}
An \emph{$(k;d,n,v)$-covering array}\footnote{Our usage of the parameters $n$ and $k$ is unfortunately swapped from the standard literature.} is a matrix in $[v]^{k \times n}$ such that every $k \times d$ submatrix is shattered. It is easy to see that if $M \in [v]^{k \times n}$ has $m < n$ submatrices that are $k \times d$ and not shattered, then a $(k;d,n-m,v)$-covering array exists, since we can just delete one column from every submatrix that is not shattered. This observation is enough to prove the following:
\begin{prop} \label{prop:cov}
For fixed $d,v\geq 2$, a $(k;d,n,v)$-covering array exists whenever
\[k \leq (1 + o(1)) \frac{(d-1)v^d}{\log_2 \beta_{v,d}} \log_2 n.\]
\end{prop}
\begin{proof}
By \cref{l21} and the above observation, a $(k;d,n,v)$-covering array exists if $n \leq n' - (1-c_v(k, d))\binom{n'}{d}$ for some positive integer $n'$. Choosing $n' = \floor{(1 - c_v(k,d))^{-1/(d-1)}}$ yields $n = \Omega((1 - c_v(k,d))^{-1/(d-1)})$. Together with monotonicity (we can freely add rows and delete columns), we get the desired after some manipulation.
\end{proof}
\cref{prop:cov}, together with bounds on the $\beta_{v,d}$, reproduce a number of results from the literature. \cref{prop:random} recovers a result of Goldbole, Skipper, and Sunley \cite{GSS} originally proved using the Lov\'{a}sz local lemma. Moreover, \cref{prop:balanced}, which generalizes to
\[\beta_{v,d} \geq \sup_{t \in \setr} {\paren*{\frac{(e^{(v-1)t}+(v-1)e^{-t})^d - e^{(v-1)dt}}{v^d}}^{-v^d}} = e\paren*{1 + \paren*{\frac{(v-1)d+1}{2} + o_v(1)}v^{-d}},\]
recovers, in the case $(v,d) = (2,3)$, results proved independently by Roux \cite{Roux} and  Graham, Harary, Livingston, and Stout \cite{GHLS}. In the general case, we reproduce a result of Franceti\'{c} and Stevens \cite{FS}. Despite being numerically the same, our result is expressed in a much simpler form, and as a result we are able to provide an asymptotic which is absent in \cite{FS}.

Finally, Das and M\'{e}sz\'{a}ros \cite{DM18} use the fact that $\gamma_{v,\infty}(1) = \prod_{i=1}^\infty (1-v^{-i})$ for prime power $v$ to construct covering arrays. However, they do not use \cref{prop:multv}, which, as previously mentioned, allows one to improve on the random construction when $v$ is a product of large prime powers.

\section*{Acknowledgments}
We thank Shagnik Das and Tam\'as M\'esz\'aros for informing
us about their paper \cite{DM18} after we posted the original version of
the present paper in the arXiv.
\cref{t12} and one of its two proofs in this original version,
as well as the connection to covering arrays, appeared
earlier in \cite{DM18}. Our second proof, \cref{t13,t13a}, and some improved
bounds for covering arrays for certain alphabet sizes are new.

\printbibliography
\end{document}